\documentclass[]{article}
\usepackage{amsfonts,amsmath,amssymb,amsthm,mathtools,mathrsfs}
\usepackage{enumitem}
\usepackage{abstract}
\usepackage[a4paper, total={6.5in, 9in}]{geometry}

\theoremstyle{definition}
\newtheorem{theorem}{Theorem}
\newtheorem{lemma}{Lemma}
\newtheorem{proposition}{Proposition}
\newtheorem{corollary}{Corollary}
\newtheorem*{definition*}{Definition}
\newtheorem{remark}{Remark}

\newtheorem{example}{Example}

\usepackage[backend=biber,style=numeric]{biblatex}
\begin{document}

\title{Perspicacious $l_p$ Norm Parameters 
}

\author{Christopher O'Neill, Vadim Ponomarenko, Eric Ren}

\maketitle

\begin{abstract}
    Fix $t\in [1,\infty]$. Let $S$ be an atomic commutative semigroup and, for all $x\in S$, let $\mathscr{L}_t(S):=\{\|f\|_t:f\in Z(x)\}$ be the "$t$-length set" of $x$ (using the standard $l_p$-space definition of $\|\cdot\|_t$). The $t$-Delta set of $x$ (denoted $\Delta_t(S)$) is the set of gaps between consecutive elements of $\mathscr{L}_t(S)$; the Delta set of $S$ is then defined by $\bigcup\limits_{x\in S} \Delta_t(S)$. Though all existing literature on this topic considers the $1$-Delta set, recent results on the $t$-elasticity of Numerical Semigroups (Behera et. al.) for $t\neq 1$ have brought attention to other invariants, such as the $t$-Delta set for $t\neq 1$, as well. Here we characterize $\Delta_t(S)$ for all numerical semigroups $\langle a_1,a_2\rangle$ and all $t\in(1,\infty)$ outside a small family of extremal examples. We also determine the cardinality and describe the distribution of that aberrant family. 
\end{abstract}

\section{Introduction}

A \emph{numerical semigroup} is a cofinite additive subsemigroup of $(\mathbb{N},+)$. Such semigroups are often denoted as an $\mathbb{N}$-semimodule generated by a set of irreducibles $a_1,...,a_k$ (usually written in increasing order), i.e., $S=\langle a_1,...,a_k\rangle = \left\{\sum\limits_{i=1}^k \lambda_ia_i:\lambda_i\in \mathbb{N}_0\right\}$. Most elements of these semigroups admit nonunique factorization into the semigroup's irreducibles; for example, we may express $60\in \langle 6,9,20\rangle$ as $0\cdot 6+0\cdot 9+3\cdot 2$ or as $4\cdot 6+4\cdot 9+0\cdot 20$. For all $x$ in a numerical semigroup $S=\langle a_1,...,a_k\rangle$, the set $\left\{(\lambda_1,...,\lambda_k):\sum\limits_{i=1}^k \lambda_ia_i=x;~\lambda_i\in \mathbb{N}_0\right\}$ is denoted by $Z(x)$ and referred to as the \emph{factorization set} of $x$. So, returning to the previous example, $(0,0,3)$ and $(4,4,0)$ are factorizations of $60$ and thus lie in $Z(60)$.

We may think of factorization sets of elements of $S=\langle a_1,...,a_k\rangle$ as equivalence classes within $k$-dimensional $l_p$-space. Invoking the usual norm, for all $t\in [1,\infty]$ and $x\in S$, we define the \emph{$t$-length set of $x$} by $\mathscr{L}_t(x)=\{\|f\|_t:f\in Z(x)\}$. Abundant literature analyses nvariants derived from the $\mathscr{L}_1$ sets of semigroups \cite{bible}. Of particular note is the $1$-\emph{delta set} of an element $x\in S$, denoted by $\Delta_1(x)$ and indicating the set of differences between consecutive elements of $\mathscr{L}_1(x)$ \cite{kaplan_2017_delta}. The \emph{$1$-delta set} of a semigroup is then defined by $\bigcup\limits_{x\in S}\Delta_1(x)$. Researchers have expended considerable effort computing this invariant for numerical semigroups \cite{kaplan_2017_delta}.

Recently, Behera et. al. \cite{ateam} computed a different norm-derived invariant, the $t$-elasticity, for all numerical semigroups and all norm parameters $t\in (1,\infty]$. This result has sparked interest in the $t$-Delta set for other norm parameters besides $1$. However, the $t$-norm for $t\neq 1$ is real-valued, rather than $\mathbb{N}$-valued, on integer vectors. Thus, the $t$-Delta set for $t\neq 1$ is extraordinarily complicated and will likely elude full characterization indefinitely. Therefore, we pursue the less lofty goal of partially characterizing these sets by determining their boundaries and limit points.

Preliminary results by Cyrusian et. al \cite{Ducks} indicate that for numerical semigroups of the form $\langle a_1,a_2\rangle$ and norm parameters $t\in [1,\infty]$, $\Delta_t(S)\subseteq [0,a_2]$. We discover that for some $t\in (1,\infty)$, $\Delta_t(S)$ is in fact dense in $[0,a_k]=[0,a_2]$. Let $t\in (1,\infty)$ be \emph{dull} if this is the case; otherwise, let $t$ be \emph{perspicacious}. Let $\mathcal{P}$ be the set of perspicacious $t$. Our main result (Theorem \ref{main}) regarding $\mathcal{P}$ is that it is countable, but dense in some ray $(a,\infty)$. A notable corollary (corollaries \ref{FLT}) is that if $t\in \mathbb{Q}\cap (2,\infty)$, then $t$ is dull.

We organize the paper leading up to Theorem \ref{main} as follows. Section 2 introduces two useful tools, the $\mu$ and $r_0$ functions, as well as some relevant properties. Section 3 shows that if $t$ meets a certain condition, $\Delta_t(S)$ is dense in an initial subinterval of $[0,a_2]$ (Proposition \ref{denselow}). Section 4 shows that in all cases, $\Delta_t(S)$ is dense in the rest of $[0,a_2]$ (Proposition \ref{densehigh}). Section 5 shows that some $t$ failing the condition from Section 3 are, in fact, perspicacious (Proposition \ref{edgecase}). Section 6 scrutinizes the set of $t$ failing the condition from Section 3. Finally, section 7 consolidates these intermediate results into the main theorem and corollary and remarks on open work.

\section{The Functions $\mu$, $r_0$, and $P$}

\subsection{$\mu$}
    The $\mu$ function proves invaluable for understanding $\mathscr{L}_t(x)$ across all $t$ and $x\in S$.

    \begin{definition*}
    \begin{enumerate}
        \item Let $\mu:[0,1]\times [1,\infty]\rightarrow \mathbb{R}$ be defined by $\mu(r,t)=\left\| \left( \frac{1-r}{a_1},\frac{r}{a_2} \right) \right\|_t$.
        \item For all fixed $t_0$, let $\mu_{t_0}:[0,1]\rightarrow \mathbb{R}$ be defined by $\mu_{t_0}(r)=\mu(r,t_0)$.
    \end{enumerate}

    \end{definition*}
 Observe that for all $x\in S$, the set equality $\text{range}(x\mu_t) = \left\{ \|\left(m,n\right)\|_t:ma_1+na_2=x;m,n\in \mathbb{R} \right\}$ holds. Because the right hand side contains $\left\{ \|\left(m,n\right)\|_t:ma_1+na_2=x;m,n\in \mathbb{N} \right\}=\mathscr{L}_t(x)$, the function $\mu$ is an object global across a semigroup that characterizes $\mathscr{L}_t(x)$ for all $t,x$. As such, we make use of properties of $\mu$ in the analysis of the sets $\mathscr{L}_t(x)$ across $S$. The following lemma collects relevant results.

\begin{lemma}
    For all $t\in [1,\infty]$:
    \begin{enumerate}
        \item $\frac{1}{a_1}=\mu_t(0)>\mu_t(1)=\frac{1}{a_2}$
        \item $\mu_t$ is smooth.
        \item If $t\neq 1$, then $\lim\limits_{r\rightarrow 0^+}\mu_t'(r)=-\frac{1}{a_1}$
        \item If $t\neq 1$, then $\lim\limits_{r\rightarrow 1^-}\mu_t'(r)=\frac{1}{a_2}$
        \item If $t\neq 1$, then $\mu_t'(r)$ is strictly increasing in $r$.

    \end{enumerate}

    Additionally, $\mu(r,t)$ is continuous in $t$.
\label{muprops}
\end{lemma}

\begin{proof}

    As claimed, $\mu_t(0)=\left\|\left(\frac{1}{a_1},0\right)\right\|_t=\frac{1}{a_1}$ and $\mu_t(1)=\left\|\left(0,\frac{1}{a_2}\right)\right\|_t=\frac{1}{a_2}$. Additionally, $\mu_t$ is a composition of smooth functions, so it is smooth. Finally, it is known that the $l_p$ norms are continuous in $p$ (\cite{lpspace}), so $\mu(r,t)$ is continuous in $t$.

     The properties pertaining to $\mu_t'$ require more effort. Suppose $t\neq 1$. Observe that
     
     $$ \mu_t'(r)=\frac{\left(\left(\frac{1-r}{a_1}\right)^t+(\frac{r}{a_2})^t\right)^{1/t}\left( \frac{-1}{a_1}(\frac{1-r}{a_1})^{t-1}+\frac{1}{a_2}(\frac{r}{a_2})^{t-1}\right)}{(\frac{1-r}{a_1})^t+(\frac{r}{a_2})^t}$$. $$\text{Substituting }\mu_t(r)=\left(\left(\frac{1-r}{a_1}\right)^t+(\frac{r}{a_2})^t\right)^{1/t}\text{ and factoring }\frac{1-r}{a_1}\text{ gives}$$

     \begin{equation} \frac{\mu_t(r)\left(\frac{1-r}{a_1}\right)^{t-1}\left(\frac{-1}{a_1}+\frac{1}{a_2}\left(\frac{ra_1}{(1-r)a_2}\right)^{t-1}\right)}{\left(\frac{1-r}{a_1}\right)^t\left(1+\left(\frac{ra_1}{(1-r)a_2}\right)^t\right)}=\mu_t(r)\left(\frac{a_1}{1-r}\right)\left(\frac{\frac{-1}{a_1}+\frac{1}{a_2}\left(\frac{ra_1}{(1-r)a_2}\right)^{t-1}}{1+\left(\frac{ra_1}{(1-r)a_2}\right)^t}\right)\label{mutprime}\end{equation}

     As $r$ approaches $0$, this approaches $\mu_t(0)a_1\frac{-1}{a_1}=-\mu_t(0)$, which, as shown, equals $-\frac{1}{a_1}$.

     Additionally, by factoring $\frac{ra_1}{(1-r)a_2}$ from equation \ref{mutprime} we obtain

     $ \mu_t(r)\frac{a_2}{r}\left(\frac{\frac{-1}{a_1}\left(\frac{(1-r)a_2}{a_1}\right)^{t-1}+\frac{1}{a_2}}{\left(\frac{(1-r)a_2}{ra_1}\right)^t+1}\right) $, which, as $r$ approaches $1$, approaches $\mu_t(1)a_2\frac{1}{a_2}=\mu_t(1)$. As shown, this equals $\frac{1}{a_2}$.

     Finally, the following expression for $\mu_t''(r)$ (derived by a CAS) is positive for $t>1$, as $r\in [0,1]$.

     $$\frac{\left(t-1\right)\left(\frac{1-r}{a_1}\right)^{t-2}\left(\frac{r}{a_2}\right)^{t-2}\left(\left(\frac{1-r}{a_1}\right)^t+\left(\frac{r}{a_2}\right)^t\right)^{1/t-2}}{a_1^2a_2^2} $$

     Therefore, $\mu_t'$ is strictly increasing for $t>1$.
\end{proof}

Notably, smoothity enables the following linear approximation of the effect of small perturbations on the length of a factorization:

\begin{lemma}
    For all $t\in (1,\infty)$, there exists $k_t$ such that for all $x\in S$ and $(m,n),(m,n)+\left(-\frac{\delta}{a_1},\frac{\delta}{a_2}\right)\in \{(m,n):ma_1+na_2=x;~m,n\in \mathbb{R}\}$, the following inequality holds: $$\left|\|(m,n)\|_t+\delta\mu_t'\left(\frac{na_2}{x}\right)-\left\|(m,n)+\left(\frac{\delta}{a_1},-\frac{\delta}{a_2}\right)\right\|_t\right|< k\frac{\delta^2}{x}$$

    \label{approx}
\end{lemma}

\begin{proof}
    Because $\mu_t$ is smooth (lemma \ref{muprops}) and defined only within the closed, bounded interval $[0,1]$, by Taylor's Theorem, there exists $k_t$ such that for all $r_1,r_2\in [0,1]$, the difference $|\mu_t(r_1)+(r_2-r_1)\mu_t'(r_1)-\mu_t(r_2)|$ is bounded by $k_t(r_2-r_1)^2$.
    
    Furthermore, by the homogeneity of the $l_p$ norm, $\|(m,n)\|_t=x\mu_t\left(\frac{na_2}{x}\right)$ and $\left\|(m,n)+\left(-\frac{\delta}{a_1},\frac{\delta}{a_2}\right)\right\|_t=x\mu_t\left(\frac{na_2+\delta}{x}\right)$. So $$\left|\|(m,n)\|_t+\delta\mu_t'\left(\frac{na_2}{x}\right)-\left\|(m,n)+\left(\frac{\delta}{a_1},-\frac{\delta}{a_2}\right)\right\|_t\right|=$$ $$ \left|x\mu_t\left(\frac{na_2}{x}\right)+x\frac{\delta}{x}\mu_t'\left(\frac{na_2}{x}\right) - x\mu_t\left(\frac{na_2+\delta}{x}\right)\right|=$$ $$ |x|\left| \mu_t\left(\frac{na_2}{x}\right)+\frac{\delta}{x}\mu_t'\left(\frac{na_2}{x}\right)-\mu_t\left(\frac{na_2+\delta}{x}\right) \right|.$$ 
    
    Applying the approximation for $\mu_t$ around $r_1=na_2$ bounds this above by $$|x|k_t\left(\frac{\delta}{x}  \right)^2=k_t\frac{\delta^2}{x}$$ as needed.
\end{proof}

This implies that the length gap $\left|\|(m,n)\|_t-\left\|(m,n)+\left(\frac{-\delta}{a_1},\frac{\delta}{a_2}\right)\right\|_t\right|$ becomes arbitrarily close to $\delta \mu_t'\left(\frac{na_2}{x}\right)$ with a uniform error bound as $x$ grows large.

\subsection{$r_0$}

For $t\in (1,\infty)$, by lemma \ref{muprops}, $\mu_t$ starts larger than $\mu_t(1)=\frac{1}{a_2}$ and decreases monotonically past $\frac{1}{a_2}$ to its single critical point before increasing monotonically back to $\frac{1}{a_2}=\mu_t(1)$. There is thus some point $r_0$ in $[0,\text{argmin}(\mu_t))$ where $\mu_t(r_0)$ is $\frac{1}{a_2}$. On $[0,r_0)$, $\mu_t$ is valued in $(\frac{1}{a_2},\frac{1}{a_1}]$; otherwise, $\mu_t$ is valued in $[\min(\mu_t),\frac{1}{a_2}]$ We locate this point, which partitions the domain of $\mu_t$ into these two regions, using its eponymous function.

\begin{definition*}
    For all $t\in [1,\infty]$, let $r_0(t):=\min\{r:\mu_t(r)=\mu_t(1)=\frac{1}{a_2}\}$.
\end{definition*}

All of the following sections make use of the property of $r_0$ outlined above and formalized below: 

\begin{lemma}
    For all $t\in (1,\infty)$ and $l\in \text{range}(\mu_t)-\{\min(\mu_t)\}$:

        If $l\in (\frac{1}{a_2},\frac{1}{a_1}]$, then $\mu_t^{-1}(l)=\{r_l\}$ for some $r_l\in [0,r_0(t))$. Furthermore, $\mu_t(r)<l$ if and only if $r>r_l$
    \label{r0}
\end{lemma}

\begin{proof}
 Let $q$ be the Holder conjugate of $t$. By \cite{ateam} and Holder's inequality \cite{lpspace}, $$\min(\mu_t)=\frac{1}{\|(a_1,a_2)\|_q}<\frac{1}{\|(a_1,a_2)\|_\infty}=\frac{1}{a_2}$$ So $r_0(t)\neq \text{argmin}(\mu_t)$. Because $\mu_t(0)=\frac{1}{a_1}>\frac{1}{a_2}>\min(\mu_t)$, by continuity of $\mu_t$ (lemma \ref{muprops}) and the intermediate value theorem, there exists $r\in (0,\text{argmin}(\mu_t))$ with $\mu_t(r)=\frac{1}{a_2}$. So by minimality of $r_0(t)$, $r_0(t)<\text{argmin}(\mu_t)$. By lemma \ref{muprops}, $\mu_t$ is strictly decreasing on $(0,\text{argmin}(\mu_t))$ and is strictly increasing on $(\text{argmin}(\mu_t),1)$.

 Now consider $l\in (\frac{1}{a_2},\frac{1}{a_1}]$. Since $\mu_t(r_0(t))=\frac{1}{a_2}<l<\frac{1}{a_1}=\mu_t(0)$ and $\mu_t$ is continuous, by the IVT, there exists $r_l\in (0,r_0(t))\cap \mu_t^{-1}(\{l\})$. As mentioned, $\mu_t$ is strictly decreasing on $(0,\text{argmin}(\mu_t)$, meaning $\mu_t$ is injective on $(0,\text{argmin}(\mu_t))$. So $r_l$ is the only member of $(0,\text{argmin}(\mu_t))\cap \mu_t^{-1}(\{l\})$.
 
    Furthermore, $\mu_t$ is increasing on $(\text{argmin}(\mu_t),1)$, so $\mu_t((\text{argmin}(\mu_t),1])=(\mu_t(\text{argmin}(\mu_t)),\mu_t(1)]$. Since $\mu_t(1)=\frac{1}{a_2}<l$, $(\text{argmin}(\mu_t),1]\cap \mu_t^{-1}(\{l\})$ is empty as well. So $\mu_t^{-1}(l)=[0,r_0(t))\cap\mu_t^{-1}(\{l\})=\{r_l\}$.

 We have just shown that if $r\in [r_0(t),1]$, then $\mu_t(r)\leq \frac{1}{a_2}<l$. Since $\mu_t$ is decreasing on $(0,r_0(t)]$, if $r\in (r_l,r_0(t))$, then $\mu_t(r)<l$ as well. So if $r\in (r_l,1]$, then $\mu_t(r)<l$. Conversely, if $r\in [0,r_l)$, because $\mu_t$ is decreasing on $[0,r_0(t))$ and $r_l<r_0(t)$, $\mu_t(r)>\mu_t(r_l)$. So $\mu_t(r)<l$ iff $r>r_l$. 
\end{proof}

We will hold discussion of other properties for section 6, where they are needed.

\subsection*{$P$}

The number $\mu_t'(r_0(t))$ splits $\Delta_t(S)$ into two sections requiring separate arguments. It also approximates elements of $\Delta_t(S)$ and identifies a family of perspicacious $t$. Let the function $P(t)$ denote this important value.

\section{Density in $[0,|a_1a_2P(t)|]$}







\begin{proposition}
    If $r_0(t)$ is irrational, then $\Delta_t(S)$ is dense in $[0,|a_1a_2P(t)|]$.\label{denselow}
\end{proposition}

\begin{proof}

    The argument of this proposition will proceed as follows. When $r_0(t)$ is irrational, we can engineer elements $x\in S$ to have both a factorization $\hat{f_1}$ at $\left(0,\frac{x}{a_2}\right)$, and a factorization $f_2$ whose $r$ parameter takes on a precisely chosen value slightly less than $r_0(t)$. By lemma \ref{r0}, $\|\hat{f}_2\|_t$ will exceed $x\mu_t(r_0(t))=\frac{x}{a_2}=\|\hat{f}_1\|_t$ by a controlled amount; moreover, $\hat{f}_2$ will lie on an injective region of the $\mu_t$ curve, where it is simple to prove that no other factorization lengths lie between $\|\hat{f_1}\|_t$ and $\|\hat{f}_2\|_t$.

    Let $r_0:=r_0(t)$, let $d\in [0,-a_1a_2P(t)]$, and let $\epsilon>0$. Because irrational multiplicative cosets of $\mathbb{Z}$ are dense in $\mathbb{R}/\mathbb{Z}$ \cite{Rudin}, for all $\epsilon'$, there exists infinitely many $N$ (and thus arbitrarily large $N$) such that $Nr_0+\frac{d}{a_1a_2P(t)}$ is at most $\epsilon'$ away from an integer, which we will denote by $z'$. Pick $\epsilon'$ and $N$ such that
    
    \begin{enumerate}
        \item $a_1a_2\epsilon'\left\|\left(\frac{1}{a_1},\frac{1}{a_2}\right)\right\|_t<\frac{\epsilon}{2}$
        \item $\frac{k_td^2}{P(t)^2a_1a_2N}<\frac{\epsilon}{2}$
        \item $ a_1a_2\epsilon'<a_1a_2+\frac{d}{P(t)}$
    \end{enumerate}

    Let $x:=a_1a_2N$ and $z:=a_1a_2z'$. By assumption (1), $xr_0+\frac{d}{P(t)}$ is at most $a_1a_2\epsilon'$ away from $z$. Furthermore, $\hat{f_1}:=(\frac{x-z}{a_1},\frac{z}{a_2})$ and $\hat{f_2}:=(0,\frac{x}{a_2})$ are integer factorizations of $x$. 

    By the above bounds, $\hat{f_1}$ lies within $$\left(\frac{x-xr_0-\frac{d}{P(t)}\pm a_1a_2\epsilon'}{a_1},\frac{xr_0+\frac{d}{P(t)}\pm a_1a_2\epsilon'}{a_2}\right) = \left(\frac{x-xr_0}{a_1},\frac{xr_0}{a_2}\right)+\left(-\frac{d}{P(t)a_1},\frac{d}{P(t)a_2} \right)\pm \left(a_2\epsilon',a_1\epsilon' \right)$$ 

    By the triangle inequality of norms, the monotonicity of the $t$-norm in this region, and the choice of $\epsilon'$,

    \begin{equation}\left|\|\hat{f_1}\|_t-\left\| \left(\frac{x-xr_0}{a_1},\frac{xr_0}{a_2}\right)+\left(-\frac{d}{P(t)a_1},\frac{d}{P(t)a_2} \right)\right\|_t\right|\leq a_1a_2\epsilon'\left\|\left(\frac{1}{a_1},\frac{1}{a_2}\right)\right\|_t<\frac{\epsilon}{2}\label{eq1}\end{equation}
    
    Furthermore, by lemma \ref{approx} and choice of $N$, \begin{equation}\left|\left\| \left( \frac{x-xr_0}{a_1},\frac{xr_0}{a_2} \right)\right\|_t+\frac{d}{P(t)}P(t)-\left\|\left(\frac{x-xr_0}{a_1},\frac{xr_0}{a_2}\right)+\left(-\frac{d}{P(t)a_1},\frac{d}{P(t)a_2} \right)\right\|_t\right|<k_t\frac{d^2}{P(t)^2x}<\frac{\epsilon}{2}\label{eq2}\end{equation}.

    Finally, \begin{equation}\|\hat{f_2}\|_t=\left\|\left(0,\frac{x}{a_2} \right)\right\|_t=\left\|\left(0,\frac{x}{a_2} \right)\right\|_t=x\left\|\left(0,\frac{1}{a_2} \right)\right\|_t=x\mu_t(r_0(t))=\left\|\left(\frac{x-xr_0}{a_1},\frac{r_0}{a_2}\right)\right\|_t\label{eq3}\end{equation} So by equations  \ref{eq1}, \ref{eq3} $$|\|\hat{f_1}\|_t-\|\hat{f_2}\|_t-d|<\left|\left\|\left(\frac{x-xr_0}{a_1},\frac{xr_0}{a_2}\right)\right\|_t-\left\| \left(\frac{x-xr_0}{a_1},\frac{xr_0}{a_2}\right)+\left(-\frac{d}{P(t)a_1},\frac{d}{P(t)a_2} \right)\right\|_t\right|-d+\frac{\epsilon}{2}$$ which, by equation \ref{eq2}, is less than $\frac{d}{P(t)}P(t)-d+\frac{\epsilon}{2}+\frac{\epsilon}{2}=\epsilon$. So $\|\hat{f_1}\|_t-\|\hat{f_2}\|_t\in (d-\epsilon,d+\epsilon)$.

    To show that $\Delta_t(x)\cap (d-\epsilon,d+\epsilon)$ is nonempty, it thus suffices to show that $\|\hat{f_1}\|_t,\|\hat{f_2}\|_t$ are consecutive lengths in $\mathscr{L}_t(x)$. Suppose there were some factorization $x\left(\frac{1-r}{a_1},\frac{r}{a_2}\right)\in Z(x)$ with $$x\left\|\left(\frac{1-r}{a_1},\frac{r}{a_2}\right)\right\|_t=\in(\|\hat{f_2}\|_t,\|\hat{f_1}\|_t)=\left(x\left\| \frac{1-\frac{z}{x}}{a_1},\frac{z}{xa_2} \right\|_t , x\mu_t(t) \right)=(x\mu_t(\frac{z}{x}),x\mu_t(r_0))$$ (see equation 3 and the definition of $\hat{f_1}$). By lemma \ref{r0}, this means that $\frac{z}{x}<r<r_0$, or that $z<xr<xr_0$. As $\frac{x-xr}{a_1}$ and $\frac{r}{a_2}$ and $\frac{x}{a_1}$ are all integers, $\frac{xr}{a_1a_2}$ must be an integer. Since $a_1a_2|z$ as well, $xr-z\geq a_1a_2$. But $z$ was constructed so that $xr_0-z<\frac{-d}{P(t)}+a_1a_2\epsilon'$, which, by choice of $\epsilon'$, is less than $a_1a_2$. So $z<xr<xr_0$ is a contradiction.

    Therefore, for all $\epsilon>0$ and $d\in (0,|a_1a_2P(t)|)$, there exists $x$ with $\Delta_t(x)\cap (d-\epsilon,d+\epsilon)\neq \emptyset$. So $\Delta_t(S)$ is dense in $[0,|a_1a_2P(t)|]$.
 \end{proof}

\section{Density in $[|a_1a_2P(t)|,a_2]$}

In this section we may relax the restriction that $r_0(t)\not\in\mathbb{Q}$.

\begin{proposition}
    For all $t\in (1,\infty)$, $\Delta_t(S)$ is dense in $[|a_1a_2P(t)|,a_2]$\label{densehigh}
\end{proposition}

\begin{proof}
    We will prove this proposition by picking factorizations differing by a trade $(a_2,-a_1)$ whose $r$ parameters are less than $r_0(t)$. This places the corresponding factorization lengths in the injective region of $\mu_t$, which implies that they are consecutive in their element's length set. By lemma \ref{approx}, we may then approximate their gap using $|a_1a_2\mu_t'|$; adjusting the $r$ parameters of the factorizations allows the gap to approach any value in $|a_1a_2\mu_t'|((0,r_0(t)))=[|a_1a_2P(t)|,a_2]$.

    Let $d\in (-a_1a_2P(t),a_2)$ and let $\epsilon>0$. By lemma 1 (monotonicity and limiting behavior of $\mu_t'$), $\mu'_t((0,r_0(t)))=(\lim\limits_{r\rightarrow 0^+}\mu_t'(r),P(t))=(-\frac{1}{a_1},P(t))$. By lemma 1, $\mu_t$ is smooth, so $\mu_t'$ is continuous. This allows for arbitrarily fine approximation of points in the image of $\mu_t'$ using rational preimages \cite{munkres}. In particular, $\frac{-d}{a_1a_2}\in (-\frac{1}{a_1},P(t))=\text{range}(\mu'_t)$.

    Accordingly, let $p\in \mathbb{Q}\cap (0,r_0)$ be such that $|\mu_t'(p)-\frac{-d}{a_1a_2}|<\frac{\epsilon}{2a_1a_2}$. Furthermore, let $\frac{M}{N}$ be a representation of $p$ such that $\frac{k_t}{N}<\frac{\epsilon}{2a_1a_2}$, where $k_t$ is from lemma \ref{approx}.

    Let $x:=a_1a_2N$. Observe that $\hat{f_1}:=x\left(\frac{1-\frac{M}{N}}{a_1},\frac{\frac{M}{N}}{a_2}\right)\in Z(x)$. Furthermore, $x(1-\frac{M}{N})\leq x(1-\frac{1}{N})=x-a_1a_2$, so $\hat{f_2}:=x\left(\frac{1-\frac{M}{N}}{a_1},\frac{\frac{M}{N}}{a_2}\right)+(a_2,-a_1)\in Z(x)$ as well. By lemma \ref{approx} and choice of $N$, $$\left|\|\hat{f_1}\|_t-\|\hat{f_2}\|_t-a_1a_2\mu'_t\left(\frac{M}{N}\right)\right|=\left|\|\hat{f_1}\|_t-\|\hat{f_2}\|_t-a_1a_2\mu'_t(p)\right|<\frac{k_t(a_1a_2)^2}{x}=\frac{k_ta_1a_2}{N}<\frac{\epsilon}{2}$$ So by choice of $p$, $$\left|\|\hat{f_1}\|_t-\|\hat{f_2}\|_t-d\right|=\left|\|\hat{f}\|_t-\|\hat{f_2}\|_t-a_1a_2\mu_t'(p)+a_1a_2\mu_t'(p)-a_1a_2\frac{-d}{a_1a_2}\right|$$
    $$\leq \left|\|\hat{f_1}\|_t-\|\hat{f_2}\|_t-a_1a_2\mu_t'(p)\right|+\left|a_1a_2\mu_t'(p)-a_1a_2\frac{-d}{a_1a_2}\right| <\frac{\epsilon}{2}+\frac{\epsilon}{2}=\epsilon$$ So $$\|\hat{f_1}\|_t-\|\hat{f_2}\|_t|\in (d-\epsilon,d+\epsilon) $$

    Now suppose there were some $\hat{f_3}=x\left(\frac{1-r}{a_1},\frac{r}{a_2}\right)\in Z(x)$ with $\|\hat{f_3}\|_t\in (\hat{f_1}\|_t,\|\hat{f_2}\|_t)$. Because $\|\hat{f_1}\|_t=x\mu(p)$ and $\|\hat{f_2}\|_t=x\mu(p-\frac{a_1a_2}{x})$ and $p<r_0$, by lemma \ref{r0}, $p-\frac{a_1a_2}{x}<r<p$. So $xp-a_1a_2<xr<xp$. However, as $\hat{f_1}$ is an integer factorization, $a_2|xp$ and $a_1|x-xp$. But as $a_1a_2|x$, this means $a_1a_2|xp$ and thus that $a_1a_2|xp-a_1a_2$. But as $\hat{f_3}$ is an integer factorization, $a_1a_2|xr$, contradicting $xp-a_1a_2<xr<xp$.

    so $\|\hat{f_2}\|_t$ and $\|\hat{f_1}\|_t$ are consecutive in $\mathscr{L}_t(x)$, meaning $\|\hat{f_1}\|_t-\|\hat{f_2}\|_t\in \Delta_t(x)$.

    So for all $d\in (-a_1a_2P(t),a_2)$ and $\epsilon>0$, there exists $x$ with $\Delta_t(x)\cap (d-\epsilon,d+\epsilon)\neq \emptyset$. So $\Delta_t(S)$ is dense in $(-a_1a_2P(t),a_2)$.
    
\end{proof}

\section{The sparse interval $[a_1,|a_1a_2P(t)|]$}

\begin{proposition}
    If $r_0(t)$ is rational and $|a_1a_2P(t)|>a_1$, then $t$ is perspicacious.\label{edgecase}
\end{proposition}

\begin{proof}
    Suppose $r_0$ is rational with simplified representation $\frac{M}{N}$. Let $d\in (a_1,|a_1a_2P(t)|)-\Lambda$, where $$\Lambda:=\left\{\frac{\lambda_1\mu_t'(r_0)}{N}+\frac{\lambda_2}{a_2}:\lambda_1,\lambda_2\in \mathbb{N}\right\}$$ Let $\epsilon:=\frac{1}{4}\min(d-a_1,|a_1a_2\mu_t'(r_0)|-d,d')$, where $d'$ is the distance from $d$ to the closest element of $\Lambda$.

    Now let $x_0$ be large enough that $\frac{k_t(a_1a_2)^2}{x_0}<\epsilon$ and let $x\geq x_0$. We will show that if $l,l'$ are consecutive in $\mathscr{L}_t(x)\cap \left[\min(\mathscr{L}_t(x)),\frac{x}{a_2}\right]$, then $|l'-l|<d-\epsilon$, and if $l,l'$ are consecutive in $\mathscr{L}_t(x)\cap \left[\frac{x}{a_2},\frac{x}{a_1}\right]$, then $|l'-l|>d+\epsilon$. We do so by examining the set $$F:=\left\{x\left(\frac{1-r}{a_1},\frac{r}{a_2}\right):x\left(\frac{1-r}{a_1},\frac{r}{a_2}\right)\in Z(x);~r\geq r_m\right\}$$ where $r_m:=\text{argmin}(\mu_t)$.

    \begin{enumerate}
        \item Let $\hat{f}\in F$ have maximal second coordinate. By maximality, $\hat{f}+(-a_2,a_1)$ cannot be a factorization, so $x-xr<a_1a_2$, meaning $1-r<\frac{a_1a_2}{x}$. Lemma \ref{approx} and the definition of $\epsilon$ then imply that  \begin{equation}\|\hat{f}\|_t\geq  x\mu_t(1)+(xr-x)\mu_t'(1)-\frac{k_t(a_1a_2)^2}{x} > \frac{x}{a_2}-a_1-\frac{k_t(a_1a_2)^2}{x}>x\mu_t(1)-a_1-\epsilon=\frac{x}{a_2}-a_1-\epsilon\label{eq4}\end{equation}. 
        \item Let $\hat{f}=x(\frac{1-r}{a_1},\frac{r}{a_2})\in F$ with $r$ minimal. By minimality, $\hat{f}+(a_2,-a_1)\not\in F$, so $r-\frac{a_1}{x}<r_m$, so $xr-xr_m<a_1a_2$. Again by lemma \ref{approx}, $$\|\hat{f}\|_t\leq x\mu_t(r_m)+(xr-xr_m)\mu_t'(r_m)+\frac{k(xr-xr_m)^2}{x}<x\mu_t(r_m)+0+\epsilon=\min(\mathscr{L}_t(x))+\epsilon$$.
        \item If $\hat{f_1}=x(\frac{1-r}{a_1},\frac{r}{a_2})\in F$ has nonmaximal second coordinate, the factorization $\hat{f_2}:=f+(-a_2,a_1)$ lies in $F$ as well. By lemma \ref{approx}, $|\|\hat{f_1}\|_t-\|\hat{f_2}\|_t|<|a_1a_2\mu_t'(r)|+\frac{k(a_1a_2)^2}{x}$. By lemma \ref{muprops}, $0=\mu_t'(r_m)<\mu_t'(r)<\mu_t'(1)=\frac{1}{a_2}$, so $\|\hat{f_1}\|_t-\|\hat{f_2}\|_t<a_1+\frac{k(a_1a_2)^2}{x}=a_1+\epsilon$.
    \end{enumerate}

    Note that by lemma \ref{muprops}, $\mu_t$ is strictly increasing on $(r_m,1)$, so $\{\|\hat{f}\|_t:\hat{f}\in F\}\subseteq \left[\min(\mathscr{L}_t(x)),\mu_t(1)\right]=\left[\min(\mathscr{L}_t(x)),\frac{x}{a_2}\right]$. So because factorizations in $F$ differ in length from their successors by at most $a_1+\epsilon$ and there exist $f\in F$ with $\|f\|_t>\frac{x}{a_2}-a_1-\epsilon$ and $f'\in F$ with $\|f'\|_t<\min(\mathscr{L}_t(x))+\epsilon$, two lengths in $\mathscr{L}_t(x)\cap \left[\min(\mathscr{L}_t(x)),\frac{x}{a_2}\right]$ differ by at most $a_1+\epsilon$, which, by definition of $\epsilon$, is less than $d-\epsilon$.

    On the other hand, consider factorizations $f=x\left(\frac{1-r}{a_1},\frac{r}{a_2}\right)$ and $f'=x\left(\frac{1-r'}{a_1},\frac{r'}{a_2}\right)$ with $\|f\|_t,\|f'\|_t$ consecutive in $\mathscr{L}_t(x)\cap \left[\frac{x}{a_2},\frac{x}{a_1}\right]$. By lemma \ref{r0}, $r'<r<r_0(t)$. Since $a_1$ divides $x-xr,x-xr'$ and $a_2$ divides $xr,xr'$, $xr\equiv xr'\mod a_1a_2$. So $xr-xr'\geq a_1a_2$, meaning, by lemma \ref{approx}, that $$\|f'\|_t-\|f\|_t>(xr'-xr)\mu'_t(r)-\frac{k_t(xr-xr')^2}{x}>-a_1a_2\mu'_t(r)-\epsilon$$ As $r<r_0$, by lemma \ref{muprops}, $\mu'_t(r)< P(t)<0$, so this is at least $|a_1a_2P(t)|-\epsilon$, which, by definition of $\epsilon$, is more than $d+\epsilon$.

    The only gap left to consider is the gap between $\max\left(\mathscr{L}_t(x)\cap \left[\min(\mathscr{L}_t(x)),\frac{x}{a_2}\right]\right)$ and $\min\left(\mathscr{L}_t(x)\cap \left[\frac{x}{a_2},\frac{x}{a_1}\right]\right)$. Let the factorizations corresponding to these lengths be $ \hat{f_1}=x(\frac{1-r}{a_1},\frac{r}{a_2})$ and $\hat{f_2}=x(\frac{1-r'}{a_1},\frac{r'}{a_2})$, respectively. We will examine this gap by examining the difference between either length and $\frac{x}{a_2}$.

    With $\hat{f_1}$, there are two cases to consider.

        \begin{enumerate}
        \item If $r>r_m$, $\hat{f_1}$ must have maximal second coordinate in $F$. As shown previously, this means $x-xr<a_1a_2$. Lemma \ref{approx} then shows that $\left|\|\hat{f}\|_t-\left(\frac{x}{a_2}-\frac{x-xr}{a_2}\right)\right|<\frac{k_t(1-r)^2}{x}<\frac{k_t(a_1a_2)^2}{x}<\epsilon$. Note that $xr$ is an integer, so $xr-x$ is also an integer, and thus an integer multiple of $\frac{1}{N}$. Therefore, $\left|\|\hat{f}\|_t-\frac{x}{a_2}\right|$ is within $\epsilon$ of an integer multiple of $\frac{1}{Na_2}$.
        \item Else, $r<r_m$. By maximality, $\hat{f}+(a_2,-a_1)$ has an $r$ parameter less than $r_0$. So $xr-xr_0<a_1a_2$. By lemma \ref{approx} $$|\|\hat{f}\|_t-x\mu_t(r_0)-(xr-xr_0)\mu_t'(r_0)|=\left|\|\hat{f}\|_t-\frac{x}{a_2}-(xr-xr_0)P(t))\right|<\frac{k_t(xr-xr_0)^2}{x}<\frac{k_t(a_1a_2)^2}{x}\epsilon$$. Note that $xr$ is an integer and $xr_0$ is a multiple of $\frac{1}{N}$, so $xr-xr_0$ is an integer multiple of $\frac{1}{N}$. So, $\left|\|\hat{f}\|_t-\frac{x}{a_2}\right|$ is within $\epsilon$ of an integer multiple of $\frac{P(t)}{N}$.
    \end{enumerate}

    There are no such cases with $\hat{f_2}$. By Lemma \ref{r0}, $r'<r_0$; by minimality, $xr'+a_1a_2>xr_0$, so $xr_0-xr'<a_1a_2$. So by lemma \ref{approx} and the choice of $x$, $$ |\|\hat{f_2}\|_t-(xr'-xr_0)P(t)-x\mu_t(r_0(t))|= \left|\|\hat{f_2}\|_t-(xr'-xr_0)P(t)-\frac{x}{a_2}\right|<\frac{k_t(xr'-xr_0)^2}{x}<\frac{k_t(a_1a_2)^2}{x}<\epsilon$$. So $\|\hat{f_2}\|t-\frac{x}{a_2}$ is within $\epsilon$ of $(xr'-xr_0)P(t)$. Since $xr'$ is an integer and $r_0$ is a multiple of $\frac{1}{N}$, $\|\hat{f_2}\|t-\frac{x}{a_2}$ is within $\epsilon$ of a multiple of $\frac{P(t)}{N}$.

    Altogether, this means that $\|\hat{f_2}\|_t-\|\hat{f_1}\|_t$ is within $2\epsilon$ of some integer linear combination of the form $\frac{\lambda_1}{Na_2}+\frac{\lambda_2P(t)}{N}$. By choice of $\epsilon$, $d$ is at least $4\epsilon$ away from all such numbers, so $\|\hat{f_2}\|_t-\|\hat{f}\|_t\not\in (d-\epsilon,d+\epsilon)$.

    Therefore, for all $x\geq x_0$, $\Delta_t(x)\cap (d-\epsilon,d+\epsilon)=\emptyset$. Therefore, $\Delta_t(S)\cap (d-\epsilon,d+\epsilon)=\left(\bigcup\limits_{x<x_0} \Delta_t(x)\right)\cap (d-\epsilon,d+\epsilon)$. Since $\bigcup\limits_{x<x_0} \Delta_t(x)$ is finite, $d$ is not a limit point of $\Delta_t(S)$. So $\Delta_t(S)$ is not dense in $[0,a_2]$, so $t$ is perspicacious.
\end{proof}

\section{Locating the edge cases}

We are interested in $$\mathcal{T}:=\{t:t\in (1,\infty);~|P(t)|>\frac{1}{a_2};~r_0(t)\in \mathbb{Q}\}\subseteq \mathcal{P}$$ We will show that $\mathcal{P}$ is countable and locate a ray $(a,\infty)$ where $\mathcal{T}$ is dense, thereby determining the exact cardinality of $\mathcal{P}$ and describing the distribution of the subset $\mathcal{T}$ of it. Doing so requires additional properties of $r_0$.

\begin{lemma}
    $r_0$ is strictly decreasing.\label{decrease}
\end{lemma}

\begin{proof}
    Suppose $t_1,t_2\in [1,\infty]$ with $t_1<t_2$. By Holder's inequality, $$\mu_{t_2}(r_0(t_1))=\left\|\left( \frac{1-r_0(t_1)}{a_1},\frac{r_0(t_1)}{a_2}\right)\right\|_{t_2}<\left\|\left( \frac{1-r_0(t_1)}{a_1},\frac{r_0(t_1)}{a_2}\right)\right\|_{t_1}=\mu_{t_1}(r_0(t_1))=\frac{1}{a_2}=\mu_{t_2}(r_0(t_2))$$ Therefore, by lemma \ref{r0}, $r_0(t_1)>r_0(t_2)$.
\end{proof}

\begin{lemma}
    $\text{range}(r_0)=[r_0(\infty),r_0(1)]=[\frac{a_2-a_1}{a_2},1]$. 
    
    \label{IVT}
\end{lemma}

\begin{proof}
    Fix $r\in (r_0(\infty),1)$. We have that $$\mu(r,1)=\frac{1-r}{a_1}+\frac{r}{a_2}=\frac{1}{a_1}+r\left(\frac{a_1-a_2}{a_1a_2}\right)>\frac{1}{a_1}+1\cdot \frac{a_1-a_2}{a_1a_2}=\frac{1}{a_2}$$ Additionally, since $r_0(\infty)<r<1$, by lemma \ref{r0}, $\mu(r,\infty)<\mu(r_0(\infty),\infty)=\frac{1}{a_2}$. By lemma \ref{muprops}, $\mu$ is continuous in $t$, and by \cite{lpspace}, $\lim\limits_{t\rightarrow \infty}\mu_t(r)=\mu_\infty(r)$, so by the Intermediate Value Theorem, there exists $t$ between $1,\infty$ such that $\mu(r,t)=\frac{1}{a_2}$, i.e., $r_0(t)=r$. So $(r_0(\infty),r_0(1))\subseteq r_0((1,\infty))$. By lemma \ref{decrease}, $r_0$ is decreasing, so $r_0((1,\infty))\subseteq (r_0(\infty),r_0(1))$, showing the reverse containment. This completes the proof of the first equality.

    Now, $\mu(1,1)=\frac{1}{a_2}$, and we have shown that for all $r<1$, $\mu(r,1)>\frac{1}{a_2}$. So $\mu(1)=1$.

    Also, $$\mu_\infty\left(\frac{a_2-a_1}{a_2}\right)=\left\|\left(\frac{1-\frac{a_2-a_1}{a_2}}{a_1},\frac{a_2-a_1}{a_2^2}\right)\right\|_\infty=\left\|\left(\frac{1}{a_2},\frac{a_2-a_1}{a_2^2}\right)\right\|_\infty$$ Since $\frac{a_2-a_1}{a_2^2}<\frac{a_2}{a_2^2}=\frac{1}{a_2}$, $$ \left\|\left(\frac{1}{a_2},\frac{a_2-a_1}{a_2^2}\right)\right\|_\infty=\frac{1}{a_2}$$. So $r_0(\infty)=\frac{a_2-a_1}{a_2}$, completing the proof of the second equality.


\end{proof}

\begin{corollary}

    \begin{enumerate}
        \item $r_0^{-1}(\mathbb{Q})$ and $\mathcal{T}$ are both countable.
        \item $r_0$ is continuous.
        \item $r_0^{-1}(\mathbb{Q})$ is dense in any ray $(a,\infty)$ where $a>1$.
    \end{enumerate}\label{r0cor}
\end{corollary}

    \begin{proof}
    
    \begin{enumerate}
        \item By lemma \ref{decrease}, $r_0$ is injective, so because $\mathbb{Q}$ is countable, $r_0^{-1}(\mathbb{Q})$ is countable. Since $\mathcal{T}\subseteq r_0^{-1}(\mathbb{Q})$, $\mathcal{T}$ is countable as well.
        \item By lemmas \ref{decrease} and \ref{IVT}, $r_0$ satisfies the IVP. By lemma \ref{decrease}, $r_0$ is injective. By \cite{Rudin}, any injective function satisfying the IVP is continuous.
        \item We know $\mathbb{Q}\cap [r_0(\infty),1]$ is dense in $[r_0(\infty),1]=[r_0(\infty),r_0(1)]$, which, by lemma \ref{decrease}, is the range of $r_0$. So because preimages under continuous functions of dense sets are dense \cite{munkres}, $r_0^{-1}(\mathbb{Q})$ is dense in $(1,\infty)$, meaning it is dense in any ray $(a,\infty)$ where $a>1$.
    \end{enumerate}
    \end{proof}

    It remains to locate $T$ where $|a_1a_2P(t)|>a_1$ if $t\in (T,\infty)$. We do so by using $\mu'_\infty(r_0(t))$ to approximate.

    \begin{lemma}
    If $0<r<\frac{a_2}{a_1+a_2}$, then $\lim\limits_{t\rightarrow \infty} \mu_t'(r)=-\frac{1}{a_1}$.\label{limlemma}
    \end{lemma}

\begin{proof}
    Fix $r\in (0,\frac{a_2}{a_1+a_2})$ and consider $\lim\limits_{t\rightarrow \infty} \mu_t'(r)$. As shown in proposition \ref{muprops} (equation \ref{mutprime}),

         \begin{equation}
         \mu_t'(r)=\frac{\mu_t(r)\left(\frac{1-r}{a_1}\right)^{t-1}\left(\frac{-1}{a_1}+\frac{1}{a_2}\left(\frac{ra_1}{(1-r)a_2}\right)^{t-1}\right)}{\left(\frac{1-r}{a_1}\right)^t\left(1+\left(\frac{ra_1}{(1-r)a_2}\right)^t\right)}=\mu_t(r)\left(\frac{a_1}{1-r}\right)\left(\frac{\frac{-1}{a_1}+\frac{1}{a_2}\left(\frac{ra_1}{(1-r)a_2}\right)^{t-1}}{1+\left(\frac{ra_1}{(1-r)a_2}\right)^t}\right)\end{equation}

    Since $r<\frac{a_2}{a_1+a_2}$, we have $\frac{ra_1}{(1-r)a_2}<1$. Also, $\lim\limits_{t\rightarrow \infty}\mu_t(r)=\mu_\infty(r)$ \cite{lpspace}. So the limit of the above expression is $$\mu_\infty(r)\left(\frac{1-r}{a_1}\right)^{-1}\frac{-\frac{1}{a_1}+0}{1+0}=\mu_\infty(r)\frac{a_1}{1-r}(-\frac{1}{a_1})$$

    Now, $\frac{ra_1}{(1-r)a_2}<1$ implies $\frac{r}{a_2}<\frac{1-r}{a_1}$, meaning $\mu_\infty(r)=\left\|\left(\frac{1-r}{a_1},\frac{r}{a_2}\right)\right\|_\infty=\max\left(\frac{1-r}{a_1},\frac{r}{a_2}\right)\frac{1-r}{a_1}$. So $$\lim\limits_{t\rightarrow \infty} \mu_t'(r)=\mu_\infty(r)\frac{a_1}{1-r}(-\frac{1}{a_1})=-\frac{1}{a_1}$$
    
\end{proof}

\begin{corollary}
    For all $r\in(0,\frac{a_2}{a_1+a_2})$, there exists $T$ such that if $t\geq T$, then $\mu_t'(r)<-\frac{1}{a_2}$.\label{limcor}
\end{corollary}

\begin{proof}
    This is true by lemma \ref{limlemma} and because $-\frac{1}{a_1}<-\frac{1}{a_2}$
\end{proof}

This naturally implies the existence of the ray.

\begin{proposition}
    There exists $T$ such that for all $t\in (T,\infty)$, $|\mu_t(r_0(t))|>\frac{1}{a_2}$. Thus, $\mathcal{T}$ is dense in some ray $(T,\infty)$. \label{bigTprop}
\end{proposition}

\begin{proof}
    We have that $$\mu_\infty\left(\frac{a_2-a_1}{a_2}\right)=\left\|\left(\frac{1-\frac{a_2-a_1}{a_2}}{a_1},\frac{a_2-a_1}{a_2^2}\right)\right\|_\infty=\left\|\left(\frac{1}{a_2},\frac{a_2-a_1}{a_2^2}\right)\right\|_\infty$$ Since $\frac{a_2-a_1}{a_2^2}<\frac{a_2}{a_2^2}=\frac{1}{a_2}$, we have $ \left\|\left(\frac{1}{a_2},\frac{a_2-a_1}{a_2^2}\right)\right\|_\infty=\frac{1}{a_2}$. So $r_0(\infty)=\frac{a_2-a_1}{a_2}$.

    Now, because $(a_2-a_1)(a_2+a_1)=a_2^2-a_1^2<a_2^2$, $r_0(\infty)=\frac{a_2-a_1}{a_2}<\frac{a_2}{a_1+a_2}$. So by lemma \ref{decrease} and part 2 of corollary \ref{r0cor}, there exists $T_1$ such that if $t\geq T_1$, $r_0(t)<\frac{a_2}{a_1+a_2}$. Let $r:=r_0(T_1)$. By corollary \ref{limcor}, there exists $T_2$ such that if $t\geq T_2$, $\mu_t'(r)<-\frac{1}{a_2}$. 
    
    Suppose $t>\max(T_1,T_2)$. By lemma \ref{decrease}, $r_0$ is decreasing in $t$, so $r_0(t)<r_0(T_1)=r$. By lemma \ref{muprops}, $\mu_t'$ is increasing in $r$, so $P(t)<\mu_t'(r)$. As $t>T_2$, $\mu_t'(r)<-\frac{1}{a_2}$. Therefore, if $t\in (T,\infty)$, then $|P(t)|>\frac{1}{a_2}$. By Corollary \ref{r0cor}, $r_0^{-1}(\mathbb{Q})$ is dense in $(T,\infty)$, so $\mathcal{T}$ is dense in $(T,\infty)$.
\end{proof}

\section{Conclusion}

We conclude with the main theorem and some closing remarks on open work.

\begin{theorem}
    For all $S:=\langle a_1,a_2\rangle$, $\mathcal{P}\subseteq r_0^{-1}(\mathbb{Q})$. It is countable, but dense in some ray $(T,\infty)$.\label{main}
\end{theorem}

\begin{proof}
    By propositions \ref{denselow} and \ref{densehigh}, if $r_0(t)\not\in \mathbb{Q}$, then $\Delta_t(S)$ is dense in $(0,|a_1a_2P(t)|)$ and $(|a_1a_2P(t)|,a_2)$, meaning $\Delta_t(S)$ is dense in $[0,a_2]$. So $\mathcal{P}\subseteq r_0^{-1}(\mathbb{Q})$. By corollary \ref{r0cor}, $r_0^{-1}(\mathbb{Q})$ is countable, so $\mathcal{P}$ is countable.
    
Furthermore, by proposition \ref{edgecase}, $$\mathcal{T}:=\left\{t:r_0(t)\in \mathbb{Q};~|P(t)|\geq \frac{1}{a_2}\right\}\subseteq \mathcal{P}$$. By proposition \ref{bigTprop}, $\mathcal{T}$ is dense in some ray $(T,\infty)$. So $\mathcal{P}$ is dense in some ray $(T,\infty)$.
\end{proof}

This theorem handles some notable special cases.

\begin{corollary}
    If $t\in (2,\infty)\cap \mathbb{Q}$, then $t$ is dull.\label{FLT}
\end{corollary}

\begin{proof}
    An elementary extension of Fermat's Last Theorem \cite{FLT} shows that there are no rational solutions to $x^t+y^y=z^t$ with $x,y,z$ rational and $t\in (2,\infty)\cap \mathbb{Q}$. Therefore, if $t\in (2,\infty)\cap \mathbb{Q}$, then $\left(\frac{1-r}{a_1}\right)^t+\left(\frac{r}{a_2}\right)^t=\left(\frac{1}{a_2}\right)^t$ only if $\frac{1-r}{a_1}$ or $\frac{r}{a_2}$ equal $0$. This means $r=1$ or $0$; by inspection, $r=1$ is the only value that works. But by lemma \ref{r0}, $r_0(t)<r_0(1)=1$, so $r_0(t)$ cannot be rational.
\end{proof}

It is, however, possible to engineer some irrational perspicacious examples. We will compute the gaps in the $t$-delta set induced by such an example.

\begin{example}

    We first show that in a general numerical semigroup $\langle a_1,a_2\rangle$, $r_0(log_{\frac{a_2}{a_1+a_2}}\left(\frac{1}{2}\right)=\left(\frac{a_2}{a_1+a_2}\right)$. Let $t_0:=log_{\frac{a_2}{a_1+a_2}}\left(\frac{1}{2}\right)$. Then $$\mu_{t_0}\left(\frac{a_2}{a_1+a_2}\right)=\left(\frac{1}{a_2}\right)\cdot\left( \left(\frac{a_2}{a_1}\left(1-\frac{a_2}{a_1+a_2}\right)\right)^{t_0}+\left(\frac{a_2}{a_1+a_2}\right)^{t_0}\right)^{1/t_0}=\left(\frac{1}{a_2}\right)\left(2\left(\frac{a_2}{a_1+a_2}\right)^{t_0}\right)^{1/t_0} $$ $$=\left(\frac{1}{a_2}\right)\left(\frac{a_2}{a_1+a_2}\right)\left(2\right)^{1/t_0}=\left(\frac{1}{a_2}\right)\left(\frac{a_2}{a_1+a_2}\right)\left(\frac{a_2}{a_1+a_2}\right)^{-1}=\frac{1}{a_2}$$ as needed.
    
    In particular, setting $S:=\langle 2,7\rangle$ and $t=log_{\frac{7}{9}}\left(\frac{1}{2}\right)$ gives $r_0(t)=\frac{7}{9}$. By inspection, $P(t)\approx -0.2298<-\frac{1}{7}$, so $a_1a_2\left|P(t)\right|>a_1$. As $r_0(t)$ is rational as well, by proposition \ref{edgecase}, $t$ is perspicacious. 

    By proposition \ref{edgecase}, if $d\in (2,|14P(t)|)-\Lambda$, where $\Lambda:=\left\{\frac{\lambda_1\mu_t'(r_0)}{9}+\frac{\lambda_2}{7}:\lambda_1,\lambda_2\in \mathbb{N}\right\}$, then $d$ is not a limit point of $\Delta_t(S)$. (Note that $14P(t)\approx 3.2172$, so this set is nonempty.) A relatively large gap in $\Lambda\cap (2,|14P(t)|)$ occurs between $\frac{58P(t)}{9}+\frac{4}{7}$ and $\frac{19P(t)}{9}+\frac{11}{7}$; accordingly, let $$d:=\frac{1}{2}\left(\left(\frac{58P(t)}{9}+\frac{4}{7}\right)+\left(\frac{19P(t)}{9}+\frac{11}{7}\right)\right)\approx 2.05446$$ Following the proof of proposition \ref{edgecase}, let $\epsilon$ be one fourth of the distance from $d$ to the closest of $2,|14P(t)|,\Lambda$, which, in this case, is $$\epsilon:=\frac{1}{4}\left( d-\left(\frac{58P(t)}{9}+\frac{4}{7}\right) \right)\approx0.0005250$$. According to the proof of proposition \ref{edgecase}, if $x$ is big enough that $\frac{k_t(2\cdot 7)^2}{x}<\epsilon$ (i.e., $\frac{196k_t}{\epsilon}<x$), then $\Delta_t(x)$ is disjoint to $(d-\epsilon,d+\epsilon)$. By \cite{Rudin}, $k_t\leq \left|\frac{\sup\{\mu''_t([0,1])\}}{2}\right|$. Numerical approximations bound this expression above by $1.5$, meaning if $x>\frac{196\cdot 1.5}{\epsilon}$ (i.e., $x>560001$), then $\Delta_t(x)$ is disjoint to $(d-\epsilon,d+\epsilon)$. Brute force computation of $\Delta_t(x)$ for $x\in [0,560001]$ shows that $\Delta_t(S)$ is completely disjoint to $(d-\epsilon,d+\epsilon)$, an interval with width approximately $0.00105$.
\end{example}

\begin{remark}
    We have shown that there is $T$ where $r_0^{-1}(\mathbb{Q})\cap (T,\infty)\subseteq \mathcal{T}\subseteq \mathcal{P}\subseteq r_0^{-1}(\mathbb{Q})\cap (1,\infty)$, but have not computed $T$ or shown whether any of these containments are strict. We also have no knowledge of $r_0^{-1}(\mathbb{Q})$ beyond its density in $(1,\infty)$, as we have no closed form for $r_0$ or $r_0^{-1}$. As such, though we know that almost all $t\in (1,\infty)$ are dull, (including the special in corollary \ref{FLT}), we cannot, in general, determine whether any particular $t$ is dull or perspicacious. 

    Notably, some algebraic manipulation reveals that for all $S$, $r_0(2)=\frac{a_2^2-a_1^2}{a_2^2+a_1^2}\in \mathbb{Q}$, and that $|\mu_2'(r_0(2))|=\frac{1}{a_2}$ exactly. Therefore, $2\not\in \mathcal{T}$, but $2\in r_0^{-1}(\mathbb{Q})$, so our results cannot determine the perspicacity of $2$.

    We also remark that experimental evidence suggests that for all $S$, $P(t)$ is increasing in $t$ on all of $(1,\infty)$. This would imply that $|P(t)|>\frac{1}{a_2}$ iff $t>2$, and thus that $\mathcal{T}=r_0^{-1}(\mathbb{Q})\cap (2,\infty)$.

    Furthermore, we suspect that $\mathcal{P}=r_0^{-1}(\mathbb{Q})$, i.e., that $t$ is perspicacious if $r_0(t)\in \mathbb{Q}$, no matter whether $|a_1a_2P(t)|>a_1$ (whether $t\in \mathcal{T}$).

    Finally, aside from some limited preliminary results by Cyrusian et. al., \cite{Ducks} the question of the limit points of $\Delta_t(S)$ for $t\in (1,\infty)$ and $S$ with higher embedding dimension remain unknown.
\end{remark}

\section*{Acknowledgements}

This work was done in extension of work begun at the 2023 SDSU Mathematics REU, supported by the National Science Foundation under Grant 1851542. Any opinions, findings, and conclusions or recommendations expressed in this material are those of the author(s) and do not necessarily reflect the views of the National Science Foundation.

\section*{References}
\printbibliography[heading=none]

\end{document}